\documentclass[10pt]{amsart}

\usepackage{amsmath,amssymb,amsfonts,epsfig}
\usepackage{epstopdf}

\newcommand{\rb}{\raisebox}
\newcommand{\ig}{\includegraphics}
\newcommand\risS[6]{\rb{#1pt}[#5pt][#6pt]{\begin{picture}(#4,15)(0,0)
  \put(0,0){\ig[width=#4pt]{#2.eps}} #3
     \end{picture}}}
\newtheorem{thm}{Theorem}[section]
\newtheorem{defn}[subsection]{Definition}

\newtheorem{lemma}[subsection]{Lemma}

\newtheorem{rem}[subsection]{Remark}
\def\st{\rb{-3pt}{\Huge $*$}}
\begin{document}

\title[A generalization of Vassiliev's planarity criterion]{A generalization of Vassiliev's planarity criterion}
\author{TYLER FRIESEN}
\date{}

\address{
Department of Mathematics, The Ohio State University,
231 West 18th Avenue, Columbus, OH 43210
{\tt friesen.15@buckeyemail.osu.edu}
}


\begin{abstract}
Motivated by his studies in knot theory V.~Vassiliev \cite{Va} introduced $X$-graphs as regular 4-valent graph with a structure of pairs of opposite edges at each vertex. He conjectured the conditions under which $X$-graph can be embedded into a plane respecting the the $X$-structure at every vertex. The conjecture was proved by V.Manturov \cite{Man}. Here we generalize these results for graphs with vertices of valency 4 or 6, \st-graphs. A problem of such generalization was posted in \cite{Sko}.
\end{abstract}

\maketitle

\section*{Introduction} \label{s:intro}

Adopting the terminology of \cite{Sko}, an {\it $X$-graph} is a regular 4-valent graph together with a splitting of the 4 half-edges at each vertex into two pairs. A choice of such splitting will be called {\it crossing structure} or {\it $X$-structure} and the half-edges at vertex in the same pair will be called {\it opposite}.
An embedding (or immersion) of a 4-valent graph into a surface induces the crossing structure at every vertex when the opposite half-edges are embedded as physically opposite lines at the vertex. Thus an embedded 4-valent graph naturally becomes an $X$-graph. An {\it $X$-embedding} of an $X$-graph into a surface is an embedding of the graph into a surface such that the induced crossing structure coincides with the given $X$-structure on the $X$-graph.
Here is an example.
$$\parbox{2in}{\tt $X$-structure:\\ four half-edges $A_1,A_2,B_1,B_2$ at the vertex are split into pairs $(A_1,A_2)$ and $(B_1,B_2)$.}\hspace{1cm}
\risS{-20}{cros}{\put(-3,41){$A_1$}\put(40,41){$B_1$}
                 \put(-10,-6){$B_2$}\put(40,-5){$A_2$}
                 \put(-10,-20){\tt $X$-embedding}}{40}{0}{0}\hspace{3cm}
\risS{-20}{cros1}{\put(-3,41){$A_1$}\put(40,41){$A_2$}
                 \put(-10,-6){$B_2$}\put(40,-5){$B_1$}
                 \put(-18,-20){\tt non $X$-embedding}}{40}{35}{45}
$$
From now on in figures we always assume that the crossing structure at a vertex is induced from the plane of a picture.

\medskip
{\bf Vassiliev's conjecture \cite{Va} (Manturov's theorem \cite{Man}.)}
{\it An $X$-graph is $X$-planar if and only if it does not contain two cycles
without common edges and with exactly one crossing vertex. A crossing vertex is a vertex which belongs to both cycles and is passed by each cycle according the crossing structure.}\\
Note that it might be other common vertices where the two cycles make a turn and do not go through along the crossing structure. 

In an expository note \cite{Sko} A.~Skopenkov posted a problem (Problem 2) of defining
\st-graphs and generalizing Manturov's theorem to them. This paper suggests an answer.

Informally, a \st-graph is a graph with a cyclic order of half-edges at each vertex considered up to reversing the order. For vertices of valency 4, this structure is equivalent to the crossing structure. An embedding of a graph into a surface induces the \st-structure from one of two possible local orientations of the surface around the vertex. This allows us to speak about \st-embeddings od \st-graphs. 
Again on figures we assume that the \st-structure is given by a counterclockwise (clockwise) cyclic order.
The formal definitions are given in Section \ref{s:def}. The formulation of our main theorem essentially repeated the Manturov's theorem. It is given in Section \ref{s:mth}. Section \ref{s:al} is devoted to auxiliary lemmas needed for the proof.

\bigskip
This work has been done as a part of the Summer 2012 undergraduate research working group
\begin{center}\verb#http://www.math.ohio-state.edu/~chmutov/wor-gr-su12/wor-gr.htm#
\end{center}
``Knots and Graphs" 
at the Ohio State University. The author is grateful to all participants of the group for valuable discussions and to the OSU Honors Program Research Fund for the financial support. Special thanks go to Sergei Chmutov, the supervisor of the group, whose generous support and guidance made this paper possible.

\section{Definitions} \label{s:def}

We use the standard graph theoretical notations from \cite{B,MT}.

\begin{defn}\rm
An \emph{unoriented cyclic order} on a finite set $S$ of cardinality $n$ is a bijection $X:S\to V(C_n)$, where $C_n$ is the cycle graph with $n$ vertices.
\end{defn}

\begin{defn}\rm
A \emph{\st-graph} is an abstract graph together with an unoriented cyclic order $X_a$ on the half-edges emanating from each vertex $a$. The unoriented cyclic order $X_a$ is called the {\it \st-structure} at $a$.
\end{defn}

A \st-structure is weaker than a usual oriented cyclic order of half-edges which is called the {\it rotation system} in \cite{MT} corresponding to embeddings of the graph into an oriented surface. On the other hand a \st-structure is stronger than an $X$-structure of opposite edges. For example, for a 6-valent vertex $a$, the \st-structure determines not only the opposite half-edge to a given half-edge as a half-edge mapped to the opposite vertex of the hexagon $C_6$, but also a pair of neighboring adjacent half-edges and a pair of remaining half-edges at ``distance" 2 in $C_6$. For a 4-valent vertex the \st-structure coincides with the $X$-structure.

\begin{defn}\rm
A \emph{\st-embedding} of a \st-graph into a surface (not necessarily orientable) is an embedding of the graph in the usual sense, with the following additional constraint: If two half-edges around a vertex map to adjacent points in the cycle graph, they must be adjacent on the surface, i.e. they must belong to the same face.
\end{defn}

Here is an example of a \st-embedding and non \st-embedding of a 6-vertex with the \st-structure given by the cyclic order $(A,B,C,D,E,F)$ into a plane.
$$\risS{-20}{v6-1}{\put(-3,35){$A$}\put(12,43){$B$}\put(35,35){$C$}
                 \put(-5,7){$F$}\put(20,-8){$E$}\put(40,10){$D$}
                 \put(-10,-20){\tt \st-embedding}}{40}{0}{0}\hspace{4cm}
\risS{-20}{v6-2}{\put(0,40){$A$}\put(28,43){$E$}\put(40,17){$C$}
                 \put(-8,17){$F$}\put(10,-8){$B$}\put(30,-3){$D$}
                 \put(-18,-20){\tt non \st-embedding}}{40}{33}{47}
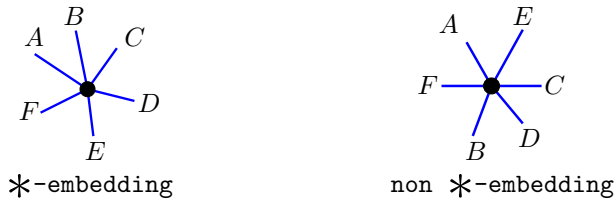
$$

\begin{defn}\rm
A \emph{crossing} between two edgewise disjoint cycles $A$ and $B$ is a common vertex $a$ which is passed by the cycle $A$ along the edges $(A_1,A_2)$ and by the cylce $B$ along the edges $(B_1,B_2)$ in such a way that the vertices $X_a(A_1)$, $X_a(A_2)$, $X_a(B_1)$, and $X_a(B_2)$ are alternate in the corresponding cycle graph $C_n$, that is they appear in the cyclic order 
$(X_a(A_1), X_a(B_1), X_a(A_2), X_a(B_2))$ in $C_n$.
%
\end{defn}

\begin{defn}\rm
A \emph{Vassiliev obstruct} is a pair of edgewise disjoint cycles with exactly one crossing.
\end{defn}

\begin{defn}\rm
The \emph{expansion} of a 6-vertex in a \st-graph is three 4-vertices with $X$-structures, arranged in either of the two ways shown.
$$\risS{-20}{v6-1}{}{40}{0}{0}
\quad\risS{-5}{tor}{}{30}{0}{0}\quad
\risS{-28}{v4-3}{}{50}{33}{47}\hspace{1.5cm}\mbox{\tt or}\hspace{1.5cm}
\risS{-20}{v6-1}{}{40}{0}{0}
\quad\risS{-5}{tor}{}{30}{0}{0}\quad
\risS{-28}{v4-3a}{}{50}{33}{47}
$$
Remember that \st-structure here is induced from the plane of the figure.

The \emph{expansion} of a \st-graph in which all vertices have order 4 or 6 is the $X$-graph generated by replacing each 6-vertex with its expansion. Note that the expansion of a graph is not uniquely defined. However, our results depend only on the existence of an expansion for any \st-graph in which all vertices have order 4 or 6, not its uniqueness.
\end{defn}

\section{Main result} \label{s:mth}

\begin{thm} 
A \st-graph in which each vertex is of order 4 or 6 has a \st-embedding into the plane if and only if it does not contain a Vassiliev obstruct.
\end{thm} 

\begin{proof}
Obviously no \st-graph containing a Vassiliev obstruct is \st-planar. It is therefore sufficient to show that any non-\st-planar \st-graph with all vertices of order 4 or 6 contains a Vassiliev obstruct. Suppose $G$ is such a graph, and let $G'$ be its expansion. Then by Lemma \ref{le1}, $G'$ is non-$X$-planar. By Manturov's theorem \cite{Man,Sko}, $G'$ contains a Vassiliev obstruct. By Lemma \ref{le2}, $G$ contains a Vassiliev obstruct.
\end{proof}

\begin{rem}\rm
The number of crossings between two cycles in the definition of a Vassiliev obstruct should be counted with multiplicities. For example, in the following graph the vertex appears twice as
a crossing of cycle $A$ with the figure-eight cycle $(B_1,B_2)$. Thus its multiplicity is 2. Therefore the cycles do not represent a Vassiliev obstruct.
$$\risS{-20}{mult}{\put(-8,50){$A$}\put(16,55){$B_1$}
                 \put(1,2){$B_2$}}{40}{45}{20}
$$
\end{rem} 

\begin{rem}\rm
In view of the recent paper \cite{Ad} 4-and 6-valent \st-graphs might be useful in knot theory as well. 
\end{rem} 

\section{Auxiliary lemmas} \label{s:al}

\begin{lemma}\label{le1}
Let $G$ be a \st-graph in which all vertices have order 4 or 6, and let $G'$ be the expansion of $G$ considered as an $X$-graph. Then if $G'$ has an $X$-embedding into the plane, $G$ also has a \st-embedding into the plane.
\end{lemma}

\begin{proof} Suppose $G'$ has an X-embedding in the plane. 
In a usual way, instead of speaking of embeddings into the plane we can speak about embeddings into a sphere.
Then for any 6-vertex in $a\in G$, consider the corresponding cycle of length 3 in $G'$. In the embedding of $G'$, this cycle is a closed loop dividing the sphere into an inside and an outside regions. At each of the three vertices on the loop, the two edges which are not part of the cycle must either both be pointing inward or both be pointing outward. (If one points inward and the other outward, it is not an X-embedding.) Suppose that there is one pair pointing in and two pairs pointing out. Then make the following transformation on the embedding of $G'$, which does not change the $X$-structure.
%
$$\risS{-20}{le1-1}{}{80}{0}{0}
\quad\risS{25}{tor}{}{30}{0}{0}\quad
\risS{-20}{le1-2}{}{80}{100}{30}
$$

Repeat this process for every expanded 6-vertex, producing an embedding of $G'$ in which every expanded 6-vertex either has nothing inside or nothing outside of it. Then to each such cycle, apply the transformation
$$\risS{-20}{le1-3}{}{80}{0}{0}
\quad\risS{15}{tor}{}{30}{0}{0}\quad
\risS{0}{le1-4}{}{40}{50}{20}
$$
producing a \st-embedding of $G$ into the sphere, and thus into the plane as well.
\end{proof}

\begin{lemma}\label{le2}
Let $G$ be a \st-graph in which all vertices have order 4 or 6, and let $G'$ be the expansion of $G$. Then if $G'$ has a Vassiliev obstruct, $G$ also has a Vassiliev obstruct.
\end{lemma}

\begin{proof} Since all of the vertices in $G$ and $G'$ have even order, the edges making up the complement to a Vassiliev obstruct can be partitioned into edgewise disjoint cycles. Thus the Vassiliev obstruct can be redefined as a partition of the edges of the graph into cycles, two of which have exactly one crossing. Note that such a partition assigns to each 4-vertex one of three possible structures:
$$\risS{-20}{st1}{}{40}{0}{0}\hspace{2cm}
\risS{-20}{st2}{}{40}{0}{0}\hspace{2cm}
\risS{-20}{st3}{}{40}{25}{25}
$$
depending on the passage of the cycles though the vertex.

\medskip
For such a structure on $G'$, call a 4-vertex in an expanded 6-vertex ``closed'' if it pairs together the edges joining it to the other two vertices in the expanded 6-vertex, ``crossing'' if it pairs together opposite edges, and ``open'' otherwise: 
$$\risS{-28}{v4-3}{\put(-20,-12){\tt Expanded 6-vertex}}{50}{33}{47}
\hspace{5cm}
\risS{-28}{cr-op-cl}{\put(-80,-12){\tt Its cycle structure. Case (10)    
                                       below.}
                     \put(31,37){\tt crossing}\put(32,22.5){\tt closed}
                     \put(-10,32.5){\tt open}}{50}{33}{47}
$$

An expanded 6-vertex must therefore have one of the following structures:
\begin{enumerate}
\item{all vertices open}
\item{all vertices closed}
\item{all vertices crossing}
\item{two open, one closed}
\item{two open, one crossing}
\item{two closed, one open}
\item{two closed, one crossing}
\item{two crossing, one open}
\item{two crossing, one closed}
\item{one crossing, one open, one closed}
\end{enumerate}

To get a Vassiliev obstruct in $G$ from a Vassiliev obstruct in $G'$, perform the following transformation to the cycle structure of each expanded 6-vertex, splitting into cases according to the 10 possibilities previously listed:
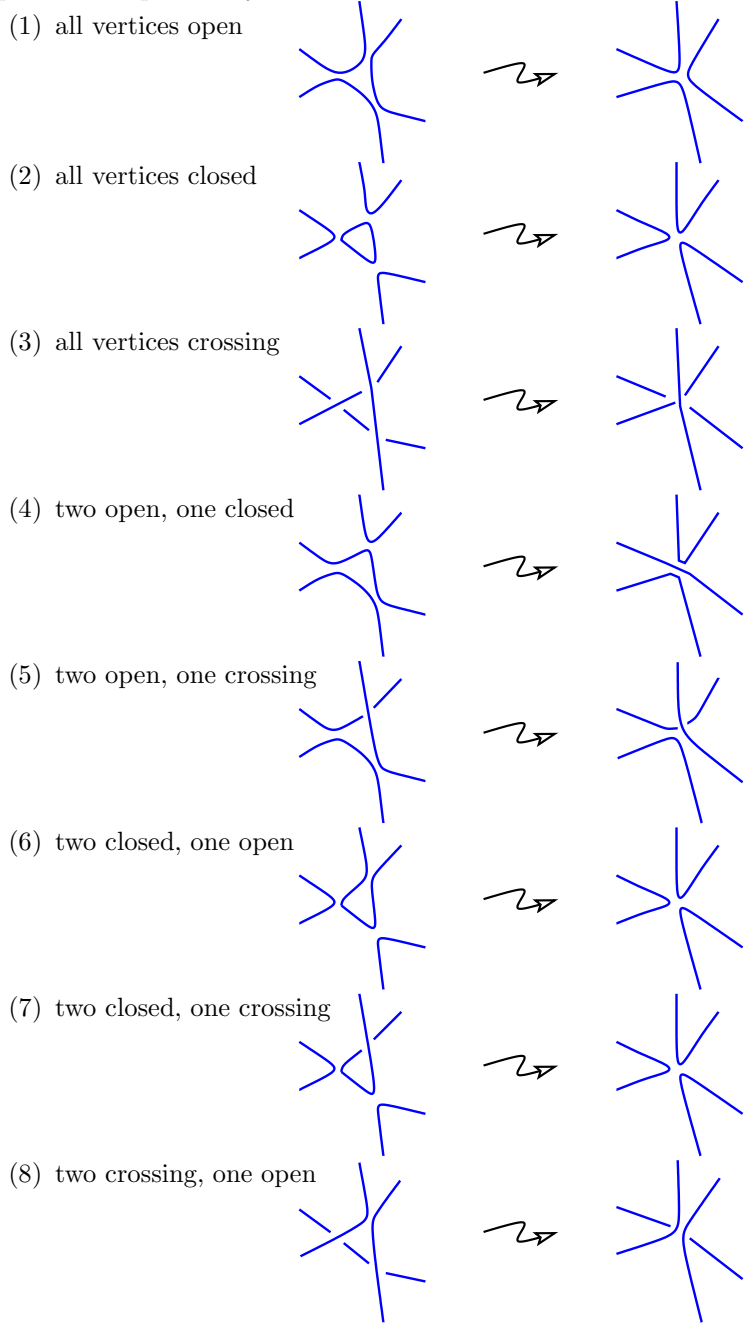
\begin{enumerate}
\item{all vertices open}
$$\risS{-15}{op-op-op-1}{}{50}{0}{0}
\qquad\risS{15}{tor}{}{30}{0}{0}\qquad
\risS{-15}{op-op-op-2}{}{50}{30}{10}
$$
\item{all vertices closed}
$$\risS{-15}{cl-cl-cl-1}{}{50}{0}{0}
\qquad\risS{15}{tor}{}{30}{0}{0}\qquad
\risS{-15}{cl-cl-cl-2}{}{50}{30}{10}
$$
\item{all vertices crossing}
$$\risS{-15}{cr-cr-cr-1}{}{50}{0}{0}
\qquad\risS{15}{tor}{}{30}{0}{0}\qquad
\risS{-15}{cr-cr-cr-2}{}{50}{30}{10}
$$
\item{two open, one closed}
$$\risS{-15}{op-op-cl-1}{}{50}{0}{0}
\qquad\risS{15}{tor}{}{30}{0}{0}\qquad
\risS{-15}{op-op-cl-2}{}{50}{30}{10}
$$
\item{two open, one crossing}
$$\risS{-15}{op-op-cr-1}{}{50}{0}{0}
\qquad\risS{15}{tor}{}{30}{0}{0}\qquad
\risS{-15}{op-op-cr-2}{}{50}{30}{10}
$$
\item{two closed, one open}
$$\risS{-15}{cl-cl-op-1}{}{50}{0}{0}
\qquad\risS{15}{tor}{}{30}{0}{0}\qquad
\risS{-15}{cl-cl-cl-2}{}{50}{30}{10}
$$
\item{two closed, one crossing}
$$\risS{-15}{cl-cl-cr-1}{}{50}{0}{0}
\qquad\risS{15}{tor}{}{30}{0}{0}\qquad
\risS{-15}{cl-cl-cl-2}{}{50}{30}{10}
$$
\item{two crossing, one open}
$$\risS{-15}{cr-cr-op-1}{}{50}{0}{0}
\qquad\risS{15}{tor}{}{30}{0}{0}\qquad
\risS{-15}{cr-cr-op-2}{}{50}{30}{10}
$$
\item{two crossing, one closed}
$$\risS{-15}{cr-cr-cl-1}{}{50}{0}{0}
\qquad\risS{15}{tor}{}{30}{0}{0}\qquad
\risS{-15}{op-op-cl-2}{}{50}{30}{10}
$$
\item{one crossing, one open, one closed}
$$\risS{-15}{cr-op-cl}{}{50}{0}{0}
\qquad\risS{15}{tor}{}{30}{0}{0}\qquad
\risS{-15}{cr-op-cl-2}{}{50}{45}{15}
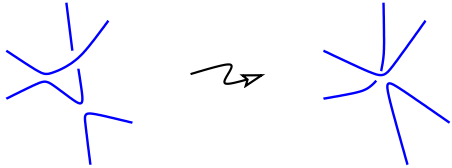
$$
\end{enumerate}
The result of this transformation is a partition of edges of the graph $G$ into cycles, as depicted in the figures on the right-hand side. We claim that among these cycles of $G$ there are two with one crossing.
Indeed in cases 1, 3, 4, 5, 6, 8, and 10, the transformation preserves the cycle segments and crossings. In case 2, a cycle is lost, but it has no crossings so it cannot be one of the two cycles with one crossing. In case 7, a crossing is lost, but it is a self-crossing so it does not affect the presence of a Vassiliev obstruct. In case 9, two crossings are lost, but since there are two of them they cannot be between the two cycles with one crossing. Thus a Vassiliev obstruct in $G'$ gives a Vassiliev obstruct in $G$.
\end{proof}

\bigskip\bigskip


\end{document}